\documentclass[11pt]{article}

\usepackage{float}
\usepackage{mathrsfs}
\usepackage{amsfonts}
\usepackage[leqno]{amsmath}
\usepackage{graphicx}
\usepackage{latexsym}
\usepackage{amsmath,amsfonts,amssymb,amsthm,mathrsfs,euscript,
color}
\usepackage{enumerate}

\newcommand{\parens}[1]{\left(#1\right)}
\def\sqr#1#2{{\vcenter{\vbox{\hrule height.#2pt
				\hbox{\vrule width.#2pt height#1pt \kern#1pt \vrule width.#2pt}
				\hrule height.#2pt}}}}
\def\5n{\negthinspace \negthinspace \negthinspace \negthinspace \negthinspace }
\def\4n{\negthinspace \negthinspace \negthinspace \negthinspace }
\def\3n{\negthinspace \negthinspace \negthinspace }
\def\2n{\negthinspace \negthinspace }
\def\1n{\negthinspace }

\def\EE{\mathsf E\:\!}
\def\PP{\mathsf P}

\def\dbF{\mathbb{F}}
\def\dbG{\mathbb{G}}

\def\dbN{\mathbb{N}}

\def\dbR{\mathbb{R}}

\def\sC{\mathscr{C}}

\def\sW{\mathscr{W}}

\def\cF{{\cal F}}
\def\cG{{\cal G}}

\def\cT{{\cal T}}

\def\ds{\displaystyle}

\def\ns{\noalign{\ss}}

\def\ss{\smallskip}

\def\q{\quad}
\def\qq{\qquad}

\def\({\Big (}
\def\){\Big )}
\def\[{\Big[}
\def\]{\Big]}

\def\a{\alpha}

\def\g{\gamma}
\def\d{\delta}
\def\e{\varepsilon}

\def\l{\lambda}

\def\si{\sigma}
\def\t{\tau}

\def\th{\theta}

\def\L{\Lambda}

\def\para#1{\vskip .25\baselineskip\noindent{\bf #1}}

\def\wt{\widetilde}

\def\bde{\begin{definition}\label}
	\def\ede{\end{definition}}
\def\be{\begin{equation}}
\def\bel{\begin{equation}\label}
\def\ee{\end{equation}}
\def\bt{\begin{theorem}\label}
	\def\et{\end{theorem}}
\def\bc{\begin{corollary}\label}
	\def\ec{\end{corollary}}
\def\bl{\begin{lemma}\label}
	\def\el{\end{lemma}}
\def\bp{\begin{proposition}\label}
	\def\ep{\end{proposition}}
\def\bas{\begin{assumption}\label}
	\def\eas{\end{assumption}}
\def\br{\begin{remark}\label}
	\def\er{\end{remark}}
\def\bex{\begin{example}\label}
	\def\ex{\end{example}}
\def\ba{\begin{array}}
	\def\ea{\end{array}}
\def\ben{\begin{enumerate}}
	\def\een{\end{enumerate}}

\def\square#1{\vbox{\hrule\hbox{\vrule height#1%
			\kern#1\vrule}\hrule}}
\def\rectangle#1#2{\vbox{\hrule\hbox{\vrule height#1%
			\kern#2\vrule}\hrule}}

\font\tenbb=msbm10 \font\sevenbb=msbm7 \font\fivebb=msbm5

\newfam\bbfam
\scriptscriptfont\bbfam=\fivebb \textfont\bbfam=\tenbb
\scriptfont\bbfam=\sevenbb

\newtheorem{theorem}{\indent Theorem}[section]
\newtheorem{definition}[theorem]{\indent Definition}
\newtheorem{proposition}[theorem]{\indent Proposition}
\newtheorem{corollary}[theorem]{\indent Corollary}
\newtheorem{lemma}[theorem]{\indent Lemma}
\newtheorem{remark}[theorem]{\indent Remark}
\newtheorem{example}[theorem]{\indent Example}

\newtheorem{assumption}[theorem]{\indent Assumption}

\makeatletter

\@addtoreset{equation}{section}
\makeatother

\def\bea{\begin{equation*}}
\def\eea{\end{equation*}}

\def\bel{\begin{equation}\label}
\def\eel{\end{equation}}

\def\ba{\begin{array}}
	\def\ea{\end{array}}
\newcommand{\ad}{&\!\!\!\displaystyle}

\def\({\Big (}
\def\){\Big )}
\def\[{\Big[}
\def\]{\Big]}
\def\q{\quad}
\def\qq{\qquad}

\def\d{\delta}
\def\e{\varepsilon}

\def\ds{\displaystyle}

\def\ns{\noalign{\smallskip}}

\def \hat {\widehat}
\def\LL{I\!\!L}
\begin{document}
	\title{\bf Exact optimal stopping  for multidimensional linear switching diffusions}
	\author{ Philip Ernst\thanks{ Department of Statistics, Rice University, Houston, 77005.   Email: {\tt philip.ernst@rice.edu} }~~~ and~~~ Hongwei Mei\thanks{Department of Statistics, Rice University, Houston, 77005. Email: {\tt hongwei.mei@rice.edu}} }
		
	\maketitle
\begin{abstract}
The paper  studies  a class of multidimensional optimal stopping problems with infinite horizon for linear  switching diffusions.  There are two main novelties in the optimal problems considered: the underlying stochastic process has   discontinuous paths and the cost function is not necessarily integrable on the entire time horizon, where the latter is often a key assumption in classical optimal stopping theory for diffusions, cf. \cite[Corollary 2.9]{PS}. Under relatively mild conditions, we show, for the class of multidimensional optimal stopping problems under consideration, that the first entry time of the stopping region is an optimal stopping time. Further, we prove that the corresponding optimal stopping boundaries can be represented as the unique solution to  a nonlinear integral equation. We conclude with an application of our results to the problem of quickest real-time detection of a Markovian drift.
	
\end{abstract}	

\para{Keywords.} Quickest detection, switching diffusions,  optimal  stopping, 
free-boundary problem.

\para{MSC 2010 Codes}: Primary: 60G40, 93C30. Secondary: 60H30, 91B70.
\section{Introduction}\label{intro}

Given a complete probability space $(\Omega,\cF,\PP)$, we consider a system of linear stochastic differential equations 
with dynamics given by
\bel{SDE-0}dX_t^i=\(a^i(\a_t)+\sum_{j=1}^2b^{ij}(\a_t)X_t^j\)dt+\sigma^i(\a_t)X_t^idW_t^i\text{ for }i=1,2,\eel
where $\a$ is a two-state 
Markov chain with state space $M:=\{1,2\}$)  and $W=(W^1,W^2)$ is a two-dimensional standard Brownian motion. 
The coefficients $a^i(\cdot),b^{ij}(\cdot),\sigma^i(\cdot):\{1,2\}\mapsto(-\infty,\infty)$  are appropriate mappings such that the solution process $X$ is non-negative (see Assumption \ref{A1} in Section \ref{sec:pre}). It is well known that equation \eqref{SDE-0} admits a unique solution $X$ and that the couple $(\a,X)$ is a strong Markov process (see, for example, \cite{Yin2009}).

The solution process $X$ of \eqref{SDE-0} coupled with the Markov chain $\a$ is generally referred to as a  \textit{regime switching} diffusion 
The primary motivation for the introduction of the Markov chain $\a$ is to  model the effect of uncertain discrete events on the underlying stochastic system (for example, the transition from a bull market to a bear market during a financial crisis). For more general properties and applications of 
regime switching diffusions, we refer the reader to \cite{Yin2009}.


The key optimal stopping problem we shall consider in the present paper is
\bel{value-0}V(\iota,x)=\inf_{\tau}\EE_{\iota,x}\int_0^\tau e^{-\int_0^t\lambda(\a_s)ds}H(\a_t,X_t)dt,
\eel
where the infimum is taken over all stopping times $\cT_{\a,X}$ of $(\a,X)$ for some appropriate discounting factor function $\lambda(\cdot):\dbN\mapsto(0,\infty)$ and for a running cost function $H:\dbN\times [0,\infty)^2 \mapsto(-\infty,\infty)$.  The key mathematical difficulties in solving this optimal stopping problem are that the paths of $(\a,X)$ are not continuous and that the cost functional (appearing on the right-hand side of equation \eqref{value-0}) may not be integrable on time horizon $[0,\infty)$. 
For general background on the theory of optimal stopping for diffusions,  we refer the reader to \cite{PS} and to the references therein.


In the one-dimensional optimal stopping literature (see, for example, \cite{Jo2006,Mc1965,ZhG2004,Zhang2005}), one usually identifies the value function and  the stopping region explicitly by finding a closed form solution to a one-dimensional free-boundary ordinary differential equation (ODE). For 
multidimensional optimal stopping problems, it is nearly impossible to find a closed form solution for the free-boundary partial differential equation (PDE) in most cases. 
In these cases, one may resort to excessive functions, Green kernel representations, and/or solutions to Hamilton-Jacobi-Bellman (HJB) equations 
(see \cite{Christ2019,Dai2018,Day2008,Day2003,Lam2013,Li2016}).  A key drawback of these characterizations of the value function is that they do not provide 
an exact optimal stopping policy. In models with switching, the continuation region and the stopping region are divided into several layers according to the state of the Markov chain and finding the explicit solution to an optimal stopping problem is not generally straightforward, even when working with one-dimensional free-boundary ODEs. This is because the continuation regions differ by layers, leading to the so-called `transition region' (i.e. the set of points belonging  to the continuation region in one layer and belonging to the stopping region in another layer). 

Unlike the aforementioned works, this paper considers a multidimensional optimal stopping problem for a diffusion under a random switching environment. 
As previously mentioned, the key mathematical difficulties in finding the optimal stopping policy are: (i) the paths of $(\a,X)$ are not continuous and (ii) the cost functional (appearing in the right-hand side in \eqref{value-0}) may not be integrable on time horizon $[0,\infty)$. It is the second difficulty which shall disable our use of a key tool in classical optimal stopping theory (\cite[Corollary 2.9]{PS}), which requires uniform integrability  on the entire time horizon (cf. \cite[(2.2.1)]{PS}).
Despite these challenges, we succeed in giving a nonlinear integral equation for the optimal stopping boundary and prove that the value function is the solution of a free-boundary PDE. 
This 
allows us to represent the optimal stopping boundary as a unique solution to a nonlinear integral equation.  One key difficulty we shall encounter is that Peskir's change-of-variable formula of local time on surfaces in \cite{Pe-2} cannot be employed due to the presence of the Markov chain. This necessitates our development of a generalized It\^o's formula for the value function acting on the switching diffusion even though the function is not  second-order continuously differentiable. 

	
	

\indent The remainder of this work is organized as follows. In Section \ref{sec:pre}, we introduce necessary preliminaries regarding the paper's key optimal stopping problem of interest given in equation \eqref{value-0}. Section \ref{sec:osb} discusses the optimal stopping boundary and its properties. Section  \ref{sec:fbp} presents the corresponding free-boundary PDE. Section \ref{sec:nie} provides the nonlinear integral equation needed to identify the optimal stopping boundary.
Section \ref{sec:exp} concludes the paper by applying the paper's results to a problem of quickest real-time detection of a Markovian drift. An analogous quickest real-time detection problem for a Brownian coordinate drift but without the incorporation of switching environments was solved in \cite{GS} for the one-dimensional case and in \cite{Er2020} for the multidimensional case.
\section{Preliminaries}\label{sec:pre}

This section introduces necessary preliminaries. Recall that $\cT_{\a,X}$ is the collection of all possible stopping times of $(\a,X)$ with respect to the natural filtration augmented with all $\PP$-null sets ($\cF=\{\cF_t:t\geq 0\}$). Since $(\a,X)$ is a L\'evy process, $\cF$ is right-continuous and the first entry time of a closed set is a stopping time if it is finite almost surely (cf. \cite[p.72]{App2009}). We recall the key optimal stopping problem under consideration disclosed in Section \ref{intro}.\\

\noindent{\bf Optimal stopping problem (OSP)}: We wish to find a stopping time $\tau_*\in\cT_{\a,X}$  such that
\bel{value}V(\iota,x)=\EE_{\iota,x}\int_0^{\tau_*} e^{-\int_0^t\lambda(\a_s)ds}H(\a_t,X_t)dt=\inf_{\tau\in\cT_{\a,X}}\EE_{\iota,x}\int_0^\tau e^{-\int_0^t\lambda(\a_s)ds}H(\a_t,X_t)dt,
\eel
subject to
\bel{SDE}dX_t^i=\(a^i(\a_t)+\sum_{j=1}^2b^{ij}(\a_t)X_t^j\)dt+\sigma^i(\a_t)X_t^idW_t^i \text{ for }i=1,2.\eel
We assume that the generator  of the continuous-time Markov Chain is  $Q=(q_{\iota\jmath})_{2\times 2}$, i.e. for $\iota\neq\jmath$, $\sum_{\jmath\in M}q_{\iota\jmath}=0$ and $ q_{\iota\jmath}\geq0$  (\cite[Section 1.4]{Yin2009}). Further, we shall assume that Assumption \ref{A1} below holds throughout the sequel.
 
\noindent \begin{assumption}\label{A1}
 	 For each $\iota\in M$, and for $i=1,2,$ we have that $\lambda(\iota)>0$,  $a^i(\iota)\geq  0$ and $\sigma^i(\iota)\neq 0$. Further, for each $\iota\in M$ and for $i\neq j$,  $b^{ij}(\iota)\geq 0$. 
 \end{assumption} 
\bigskip
\noindent Standard results for stochastic differential equations imply that under Assumption \ref{A1}, equation \eqref{SDE} admits a unique solution $X=(X^1,X^2)$  in $[0,\infty)^2$ for any non-negative initial state. Note that the regime switching diffusion $(\a,X)$ is a homogeneous strong Markov  process. Let $X(\iota,x)=(X^1(\iota,x),X^2(\iota,x)) $ denote the solution of \eqref{SDE} with initial values $(\a_0,X_0)=(\iota,x)$.
For all initial values $(\iota,x)\in M\times(0,\infty)^2$, the solution process $X$ stays positive with probability 1, i.e. for $t>0$, $$\PP_{\iota ,x}\big(X_t\in(0,\infty)^2
\big)=1.$$
For each $\iota\in M$, the following comparison principle (\cite{Fe2000}) holds:
\bel{com}\text{ If }x_1\leq y_1\text{ and } x_2\leq y_2,\,\, X_t^i(\iota,x)\leq X^i_t(\iota,y).\eel 
%
Further, the infinitesimal generator $\LL$ of $(\a,X)$ is non-degenerate on $(0,\infty)^2$ and has the following form  for any $f(\iota,\cdot)\in\sC^2\big((0,\infty)^2\big)$ (cf. (2.4) in \cite{Yin2009}) 
 \bel{gen} \LL f(\iota, x) =\sum_{i=1}^2\(a^i(\iota)+\sum_{j=1}^2b^{ij}(\iota)x_j\)\partial_if(\iota, x)+\frac12\sum_{i=1}^2(\si^i(\iota)x_i)^2\partial_{ii}f(\iota, x)+\sum_{\jmath\in M}q_{\iota\jmath}f(\jmath,x).\eel
Therefore, $(\a,X)$ is strong Feller and admits a transition density $p_t(\iota,x;\jmath,y)$ on $M\times(0,\infty)^2$ 
 as follows $$\PP_{\imath,x}(\a_t=\jmath,X_t\in dy)=p_t(\iota,x;\jmath,y)dy.$$ 
 \indent In the sequel, we shall also assume that Assumption \ref{A2} below holds.

\ss

\begin{assumption}\label{A2} For each $\iota\in M$, $H(\iota,\cdot)$ 
 is concave on $[0,\infty)^2$, locally Lipschitz up to the boundary of   $[0,\infty)^2$,  and is increasing with respect to each variable. Moreover
 $$\lim_{z\rightarrow\infty}H(\iota,(z,0))=\lim_{z\rightarrow\infty}H(\iota,(0,z))=\infty.$$
\end{assumption}
\bigskip
\noindent Under Assumption \ref{A2}, $H$ is bounded from below and therefore the value function in \eqref{value} is well-defined. 
Note that  \eqref{value} is Lagrangian formulated (\cite{GS}). \\
\indent We now pause to introduce some necessary notation for future use$$\underline H:=\inf_{\iota\in M}H(\iota,(0,0)), \q\underline \lambda :=\inf_{\iota\in M}\l(\iota),
\q\text{and}\q\underline a^i :=\inf_{\iota\in M}a^i(\iota).$$
Without loss of generality, we may assume that $\underline H<0$. Otherwise, $\tau=0$ is always optimal for the optimal stopping problem in \eqref{value} and the problem is trivial. Under Assumption \ref{A2}, the inequality \bel{infinitehorizon}\EE_{\iota,x}\int_0^\infty e^{-\int_0^t\lambda(\a_s)ds}H(\a_t,X_t)dt<\infty, \eel 
does not necessarily hold. It is important to note that the inequality in \eqref{infinitehorizon} is analogous to Assumption \cite[(2.2.1)]{PS}, which is needed for \cite[Corollary 2.9]{PS} to hold.\\
\indent In the sequel, we shall denote by $\tau_D$ the first entry time to the set $D$ for the stochastic process $(\a,X)$, i.e.  
\begin{equation}\label{tauD}
\tau_D:=\inf\{t\geq 0: (\a(t),X(t))\in D\},
\end{equation}
where by definition $\inf\emptyset=\infty$. We shall further assume throughout that Assumption \ref{A3} below holds. 
 \begin{assumption}\label{A3} For $i=1,2$, the solution of $\wt X_t^i$ to  $$d\wt X_t^i=(a^i(\a_t)+b^{ii}(\a_t)\wt X_t^i)dt+\sigma^i(\a_t)\wt X_t^idW_t^i$$ 
 	is recurrent to $(c,\infty)$ for any $c>0$.
 \end{assumption}
 \bigskip
\noindent Proposition \ref{veriA3}, which we shall present at the conclusion of this section, shall show that Assumption \ref{A3} may be easily verified. 

\subsection{One-dimensional optimal stopping problems}\label{sec:1d}
We now introduce a one-dimensional optimal stopping problem which shall prove useful for later study.
For any positive real number $N$, we consider the one-dimensional optimal stopping problem
\bel{1dstop}\ba{ll}\ds\wt V^{1,N}(\iota,z)\ad=\inf_\t\EE_{\iota,z} \int_{0}^\t e^{-\int_0^t\lambda(\a_s) ds}H^N(\a_t, (\wt X^1_t,0))dt,\ea\eel
where 
\begin{equation}\label{HN}
H^N(\iota,x)=\min(H(\iota,x), N)
\end{equation}
and where $\wt X^1$ solves the stochastic differential equation 
$$d\wt X_t^1=(a^1(\a_t)+b^{11}(\a_t)\wt X_t^1)dt+\sigma^1(\a_t)\wt X_t^1dW_t^1.$$ 
By the comparison principle in \eqref{com}, we have, for all $t\geq 0$, that if $\wt X_0^1= X_0^1$ then $\wt X_t^1\leq X_t^1$.
 The optimal stopping problem in \eqref{1dstop} is thus a special case of the optimal stopping problem in \eqref{value} when the second coordinate process is fixed at zero. We now consider the optimal stopping region for $\wt V^{1,N}$ in Proposition \ref{1-dtime} below.
 
 \begin{proposition}\label{1-dtime} Suppose   Assumptions  \ref{A1}, \ref{A2} and \ref{A3} hold, and let $N>0$ be large.
 Then the following two statements hold: \ss
 	
 	 {\rm (1)} The optimal stopping region for  $\wt V^{1,N}$ is $\{(\iota,z)\in M\times (0,
 	 \infty):\wt V^{1,N}(\iota,z)=0\}$. 
 	 \ss

 	 {\rm (2)}  For the one-dimensional optimal stopping problem in \eqref{1dstop}, the optimal stopping region for $\wt V^{1,N}$ has the form $$\{(\iota,z)\in M\times(0,\infty): z\geq \bar x_\iota^{1}\},$$ where for $\iota\in M$, $\bar x_\iota^{1}$ are finite positive real numbers. Further, the first entry time of the stopping region is finite almost surely.

 \end{proposition}

\begin{proof}
	(1) $H^N$ is bounded and $(\a,X)$ is a Feller process. This follows immediately from \cite{Dai2018}.\ss
	
(2) Note $\wt X^1_t(\iota,x_1)$ is affine in $x_1$ and that $H^N(\iota,\cdot,0)$ is concave increasing. Further, $\wt V^N(\iota,\cdot)$ is also concave increasing on $(0,\infty)$. This suggests that for the optimal stopping boundaries $\{ \bar  x_\iota^{1}\}_{\iota\in M}$ that the stopping region for \eqref{1dstop} should be $\{(\iota,z)\in M\times(0,\infty):z\geq \bar x_\iota^{1}\}.$ We also note that $\bar x_\iota^{1}$  is dependent on $N$ and, for each $\iota\in M$, $\bar x_\iota^{1}$ is decreasing as $N$ increases. That is, if $z>\bar x_\iota^{1}$, $\wt V^{1,N}(\iota,z)=0$; otherwise $\wt V^{1,N}(\iota,z)<0$. We now wish to show, for $N>0$ (large) with $N>N_0$, that $\bar x_\iota^{1}<\infty$ for each $\iota\in M$. We prove this by contradiction.



First, recall that for each $\iota\in M$,
$\wt V^{1,N}(\iota,\cdot)$ is increasing, concave and continuous on $[0,\infty)$. 
For the sake of contradiction, let us suppose that $\bar x_{1}^1=\infty$.  This means that
\bel{allnegative}\wt V^{1,N}(1,z)<0 \text{ for all } z\in[0,\infty).\eel  Let $\tau_*$ denote the first entry time of the stopping region and let $\tau_1:=\inf\{t\geq 0:\a_t\neq 1\}$ be the first transition time of the Markov chain (which is independent of $N$ and $z$). We thus have that $\tau_1\leq\tau_*$ and
$$\ba{ll}\wt V^{1,N}(1,z)
\ad\geq \EE_{1,z} \int_{0}^{\tau_1}e^{-\int_0^t\lambda(\a_s)ds}\(H^N(\a_t,(\wt X^1_t,0))-\underline H\)dt+\underline H/\underline\lambda.
\ea$$
Note that the right-hand side is increasing in $N$ and $z$. By Assumption \ref{A2}, the right-hand side tends to $\infty$ when both $N\rightarrow\infty$ and $z\rightarrow\infty$.
Therefore we may take large $N_0$ and $z_0$ such that $\wt V^{N_0}(1,z_0)>0$. Then for large $N$ with $N>N_0$, it follows that $\wt V^N(1,z)>0$ for $z>z_0$, contradicting \eqref{allnegative}. Thus, by contradiction,  we conclude that $\bar x_1^{1}<\infty$. One may similarly prove that, for each $\iota\in M$, $\bar x_\iota^{1}<\infty$.
By Assumption \ref{A3}, $(\a,\wt X^1)$ will enter the stopping region in finite time. This completes the proof.
\end{proof}	

Of course, the above results for the one-dimensional optimal stopping problem in \eqref{1dstop} can be employed when working with the one-dimensional optimal stopping problem 
\bel{1dosp2}\ba{ll}\ds\wt V^{2,N}(\iota,z)\ad=\inf_\t\EE_{\iota,z} \int_{0}^\t e^{-\int_0^t\lambda(\a_s)ds}H^N(\a_t, (0,\wt X^2_t))dt,\ea\eel
where $\wt X^2$ solves 
$$d\wt X_t^2=(a^2(\a_t)+b^{22}(\a_t)\wt X_t^2)dt+\sigma^2(\a_t)\wt X_t^2dW_t^2.$$
The one-dimensional optimal stopping problems in \eqref{1dstop} and in \eqref{1dosp2} shall prove helpful in our later study of multidimensional optimal stopping problems. We conclude the present section by showing, as promised earlier, that Assumption \ref{A3} may be easily verified.

\begin{proposition}\label{veriA3} The following two statements hold. \\ {\rm (1)} Assumption \ref{A3} holds for $i=1,2$ if 
	 for any $c>0$, there exists a twice continuously differentiable function $ W_c:M\times(0,c]\mapsto(-\infty,\infty)$ such that 
$\lim_{z\rightarrow0^+}W_c(\iota,z)=\infty$ and $$ (a^i(\iota)+b^{ii}(\iota)z)W_c'(\iota, z)+\frac12(\si^i(\iota)z)^2W_c''(\iota, z)+\sum_{\jmath\in M}q_{\iota\jmath}W_c(\jmath,z)\leq 0$$
for $(\iota,z)\in M\times (0,c).$

{\rm (2)}  Assumption \ref{A3} holds if   \bel{H5}a^i(\cdot)>0 \text{ and  }b^{ii}(\cdot)\geq0.\eel 
\end{proposition}
\begin{proof}
(1) The proof is identical to that of \cite[Theorem 3.14]{Yin2009}. \ss

(2) Consider $\tau_c:=\inf\{t\geq 0: \wt X^1_t\geq c\}$. Since $\wt X^1_t$ never returns to $0$,  It\^o's formula  yields that
 \bea\ba{ll}\ds c\geq\EE_{\iota,x_1} \wt X^1_{\tau_c\wedge N}-x_1\geq\EE_{\iota,x_1}\int_0^{\tau_c\wedge N}a^{1}(\a_s)dt\geq \underline a^{1} \EE_{\iota,x_1}(\tau_c\wedge N).\ea\eea	
 We thus have that
 $\EE_{\iota,x_1}(\tau_c\wedge N)\leq c/\underline a^{1}$. Letting $N\rightarrow\infty$, we obtain
 $$\PP_{\iota,x_1}(\tau_c\geq N)\leq \frac{c}{N\underline a^{1}}\rightarrow 0.$$
Thus $\PP_{\iota,x_1}(\tau_c<\infty)=1$ and Assumption \ref{A3} holds.\end{proof}
 	

\section{Optimal stopping boundary of the optimal stopping problem in \eqref{value}} \label{sec:osb}
In this section, we will derive some helpful properties of the optimal stopping boundary for the optimal stopping problem in \eqref{value}. Throughout the sequel, we shall again assume that Assumptions \ref{A1}, \ref{A2}, and \ref{A3}  hold.\ss

The candidate continuation region $C$ and stopping region $D$ for the optimal stopping problem in \eqref{value} are, respectively,
\bel{contr}\ba{ll} C\ad=\Big\{(\iota,x)\in M\times(0,\infty)^2\!: V(\iota,x)<0\Big\},
\ea\eel
and 
\bel{stopr}\ba{ll}D\ad=\Big\{(\iota,x)\in M\times(0,\infty)^2\!: V(\iota,x)=0\Big\}
.\ea\eel
 $C$ and $D$ are divided into different layers by the discrete states of the Markov chain. Let $C_{\iota}$ and $D_{\iota}$ be the $\iota$th layer of $C$ and $D$, respectively. $C_{\iota}$ and $D_{\iota}$ are subsets of $(0,\infty)^2$, i.e.
\bel{contr-i}\ba{ll} C_\iota=\Big\{x\in (0,\infty)^2\!: V(\iota,x)<0\Big\}
\text{ and }D_\iota=\Big\{x\in (0,\infty)^2\!: V(\iota,x)=0\Big\}
.\ea\eel

Given the formulation of $D$ in  \eqref{stopr} above, a natural question to ask is as follows: is $\tau_D$, the first entry time of the stopping region $D$, the optimal stopping time for the optimal stopping problem in \eqref{value}? We shall soon see that the answer will be `yes', but the solution will be far from trivial because Corollary 2.9 from \cite{PS} cannot be invoked since 
the inequality in  \eqref{infinitehorizon} does not hold. To overcome this difficulty, a key idea we shall employ is to cut off the function $H$ from above  by a positive constant $N$ and then to let $N\rightarrow\infty$.  We shall refer to this idea 
as an `approximation' stopping problem. 

\subsection{Approximation method} \label{sec31}
\indent \indent  To overcome the difficulties mentioned above, we turn to the so-called `approximation' optimal stopping problem in \eqref{OSPN} below.  

\noindent {\bf Optimal Stopping Problem (OSP)}-$N$: We wish to find a stopping time $\tau_*\in\cT_{\a,X}$ such that
\bel{OSPN} V^N(\iota,x)=\EE_{\iota,x}\int_0^{\tau_*}e^{-\int_0^t\lambda(\a_s)ds}H^N(\a_t,X_t)dt=\inf_{\tau\in\cT_{\a,X}}\EE_{\iota,x}\int_0^{\tau}e^{-\int_0^t\lambda(\a_s)ds}H^N(\a_t,X_t)dt,\eel
where $H^N(\iota,x)$ has been defined in \eqref{HN} above. 
 Since $H^N$ is bounded  and $(\a,X)$ is a Feller process, Proposition \ref{prop31} below immediately follows from \cite{Li2016}.
\begin{proposition}\label{prop31}
	$ V^{N}$ is  continuous  and is the unique viscosity solution to the following HJB equation
	\bel{HJB}\min( \LL w-\l w+H^N, -w)=0,\eel
	where $w: M\times(0,\infty)^2\mapsto\dbR$.
	
\end{proposition}

We now let $N\rightarrow\infty$ for the optimal stopping problem in \eqref{OSPN}. Let $\Lambda_t:=\int_0^t\l(\a_s)ds$ and	 $$\left\{\ba{ll}\ad C^{N}:=\Big\{(\iota,x): V^{N}(\iota,x)<0\Big\},~C_\e^{N}:=\Big\{(\iota,x): V^{N}(\iota,x)<-\e\Big\},\\
\ns\ad D^{N}:=\Big\{(\iota,x): V^{N}(\iota,x)\geq 0\Big\},~ D_\e^{N}:=\Big\{(\iota,x): V^{N}(\iota,x)\geq -\e\Big\}.\ea\right.$$
Since $V^N$ is continuous, $C^N$ and $C_\e^N$ are open sets and $D^N$ and $D_\e^N$ are closed sets. Moreover,   $D_\e^N$ is increasing as $N$ increases and is decreasing as $\e$ decreases. We now proceed with some necessary results on the finiteness of the  first entry time into the stopping region $D^N$.
\begin{lemma}\label{finitetau}
	The first entry time $\tau_{D^N}$ of $D^N$ is finite almost surely. Therefore $\tau_D$ and $\tau_{D_\e^N}$ are finite almost surely.
\end{lemma}

\begin{proof}
	By the monotonicity of $H$, we have 
	$V\geq V^N\geq\wt V^{1,N}$, where $\wt V^{1,N}$ has been defined in \eqref{1dstop}. This means that
	$$D\supset D^N\supset \{(\iota,x_1,x_2)\in M\times(0,\infty)^2: x_1\geq \bar x_\iota^{1}\}.$$
By Proposition \ref{1-dtime}, $\tau_{D^N}$ is finite almost surely. Noting that $\tau_{D_\e^N},\tau_D\leq \tau_{D^N}$, the proof is now complete.
	\end{proof}

\subsection{The optimal stopping region for the optimal stopping problem in \eqref{value}}
The purpose of this section is to present the optimal stopping region for the optimal stopping problem in \eqref{value} by letting $N\rightarrow\infty$ for the optimal stopping problem in \eqref{OSPN}. We begin with Lemma \ref{thmops} below, which assures that the first entry time is optimal for the optimal stopping problem in \eqref{value}.
\begin{lemma}\label{thmops}  If there exists a $N_0>0$ such that $\t_{D^{N_0}}$ is finite almost surely with
	\bel{optrule}\EE_{\iota,x}\int_0^{\t_{D^{N_0}}}e^{-\L_t }|H(\a_t,X_t)|dt<\infty\text{ for any }(\iota,x)\in M\times (0,
	\infty)^2,\eel
 then $\t_{D}$ is the  optimal stopping time of the optimal stopping problem in \eqref{value}, provided that $\partial D$ is Green's regular.
 
	\end{lemma}
\begin{proof}
Immediately, we see that $D^{N}\subset D^{N}_\e$. Thus, with probability 1, $\t_{D^{N}_\e}\leq \t_{D^{N}}$.  Let $\t_j$ denote the first exit time of $(\a,X)$ from the set  $M\times(1/j,j)^2$.
	 Further, note that $C_\e^{N}$ is an open subset of $C^{N}$ and that the closure of $C_\e^{N}$ is a subset of $C^N$. Employing Snell's envelope (cf. \cite[Theorem 2.4]{PS}), we see that the following process is a martingale
	$$e^{-\Lambda_ {t\wedge \t_j\wedge\t_{D^{N}_\e}}} V^{N}( \a_{t\wedge \t_j\wedge\tau_{D^{N}_\e}}, X_{t\wedge \t_j\wedge\tau_{D^{N}_\e }})+\int_0^{t\wedge \t_j\wedge\tau_{D^{N}_\e}} e^{-\L_ s}H^N(\a_s, X_s)ds.$$

\noindent We proceed to calculate
	\bel{opse}\ba{ll}\ad V^{N}(\iota,x)=\EE_{\iota,x}\[\int_0^{t\wedge \t_j\wedge\t_{D^{N,}_\e}} e^{-\L_ s}H^N( \a_s,X_s)ds+e^{-\Lambda_ {t\wedge \t_j\wedge\t_{D^{N}_\e}}} V^{N}( \a_{t\wedge \t_j\wedge\t_{D^{N,}_\e }},X_{t\wedge \t_j\wedge\t_{D^{N}_\e }})\]\\
	\ns\ad=\EE_{\iota,x}\[\int_0^{\t_j\wedge\t_{D^{N}_\e}}e^{-\L_s } H^N(\a_s,X_s)ds+e^{-\Lambda_ { \t_j\wedge\t_{D^{N}_\e}}} V^{N}(\a_{ \t_j\wedge\t_{D^{N}_\e }}, X_{ \t_j\wedge\t_{D^{N}_\e }})\]\\
	\ns\ad=\EE_{\iota,x}\[\int_0^{\t_{D^{N}_\e}}e^{-\L_ s} H^N( \a_s,X_s)ds+e^{-\Lambda_ {\t_{D^{N}_\e}}} V^{N}(\a_{ \t_{D^{N}_\e }},X_{ \t_{D^{N}_\e }})\]\\
	\ns\ad\geq \EE_{\iota,x}\[\int_0^{\t_{D^{N}_\e}}e^{-\L_ s} H^N( \a_s,X_s)ds\]-\e, \ea\eel	
	where in the second line of the above display we have let $t\rightarrow\infty$ and have invoked the fact that $  V^{N}$ is continuous and bounded. Further, in the third line of the above display, we have (i) let $j\rightarrow\infty$ 
and recalled that for $N>N_0$, $\t_{D_\e^{N}}\leq\tau_{D^N}$ is finite almost surely (Proposition \ref{finitetau}) (ii) applied monotone convergence and (iii) employed the fact that $D^{N}_\e
	$  is closed.

	 Provided \eqref{optrule} holds, it follows by \eqref{opse} that \bel{difV}\ba{ll}V(\iota,x)\ad\geq V^{N}(\iota,x)\geq\EE_{\iota,x}\[\int_0^{\t_{D_\e^{N}}}e^{-\L_ t}H^N( \a_t,X_t)dt\]-\e\\
	\ns\ad \geq V(\iota,x)-\e-\EE_{\iota,x}\int_0^{\t_{D_\e^{N}}}e^{-\L_ t}|H^N(\a_t,X_t)-H( \a_t,X_t)|dt.\ea\eel	
	Sequentially letting $N\rightarrow\infty$ and $\e\rightarrow0^+$, dominated convergence yields that
	\bel{VN}V(\iota,x)=\lim_{N\rightarrow \infty}V^{N}(\iota,x)=\lim_{\e\rightarrow0^+}\lim_{N\rightarrow\infty}\EE_{\iota,x}\int_0^{\t_{D^N_\e}}e^{-\L_t}H( \a_t,X_t)dt.\eel
	%
Further, it is useful to note that
	$$\ba{ll}\ds\bigcup_{N} D_\e^{N}= \bigcup_{N}\{V^{N}\geq -\e\}\subset\{V\geq -\e\}=D_\e,\ea$$
	and 
	$$\ba{ll}\ds\bigcup_{N} D_\e^{N}=\bigcup_{N}\{V^{N}\geq -\e\}\supset\{V\geq 0\}=D.\ea$$
\indent Recall that $(\a,X)$ has right-continuous paths with left-limits (\cite[Proposition 2.4]{Yin2009}). By Theorem 4 in \cite[Section 2.4]{Chung2013}, 
we have that
$$	V(\iota,x)=\EE_{\iota,x}\int_0^{\t_{D}}e^{-\L_t}H( \a_t,X_t)dt$$
i.e. $\tau_D$ is the optimal stopping time. This concludes the proof.
\end{proof}
Lemma \ref{thmops} may be considered an extension of \cite[Corollary 2.9]{PS} in the setting where  (i) the cost functional  is not necessarily integrable on the entire time horizon and (ii) the underlying stochastic process does not have continuous paths. This extension enables us to work with a larger class of multidimensional optimal stopping problems. Before continuing to discuss properties of the stopping and continuation regions in Section \ref{propstop}, we pause to offer some remarks below.
\begin{remark}  
\rm (1) Since $H$ is bounded from below,  \eqref{optrule} is equivalent to
	$$\EE_{\iota,x}\int_0^{\t_{D^{N_0}}}e^{-\L_ t}H(\a_t,X_t)dt<\infty$$
	for any $(\iota,x)\in M\times (0,
	\infty)^2$.
	Further, we have that $D^N\subset D\subset D_\e$. This means that we do not require uniform integrability on the interval $(0,\infty)$. Rather, we only require uniform integrability on a (random) subinterval $(0,\tau_{D^{N_0}})$. \ss
	
\rm \indent (2)	  If $\lambda(\cdot)>0$ and $H$ is bounded on $C^{N_0}$, then \eqref{optrule} holds. We will see this in the next section.
	
	
	

\end{remark}

\subsection{Properties of the stopping and continuation regions}\label{propstop}
This section considers properties of the stopping regions $D$ and continuation regions $C$. The key result of this section is Theorem \ref{OPSStop}. We first turn to a few preparatory lemmas and propositions required for the proof of this theorem. 

\begin{lemma}\label{hatVpp}
	For each $\iota\in M$,	$V(\iota,x)$ is increasing in $x$. Moreover $ V(\iota,\cdot)$ is  concave 
	and continuous on $(0,\infty)^2$.  The same results also hold for $V^N$.
\end{lemma}
\begin{proof} Let us fix a value of $\iota\in M$. Invoking the comparison principle in \eqref{com},  we see that $ V(\iota,x)$ is increasing in $x$ since $H(\iota,x)$ is also increasing in $x$. Note that  $X^{\iota,x}_t$ is affine with respect to $x$ and that, for each $\iota\in M$, $ V(\iota,\cdot)$ is concave. Consequently $ V(\iota,\cdot)$  is continuous on $(0,\infty)^2$. The aforementioned results also hold for $V^N$ because $H^N$ is concave. This completes the proof.
\end{proof}
We now proceed to define
$$b_{\iota}(x_1):=\inf\Big\{x_2>0:  V(\iota,x_1,x_2)=0\Big\}.$$
The curve $S_{b_\iota}$\,
gives the optimal stopping boundary on the $\iota$th layer.  
One may similarly define $b_\iota^N$ as the optimal stopping boundary for $S^N$. We now prove in Lemma \ref{finiteboundary} below that the optimal stopping time $\tau_D$ given in \eqref{tauD} is finite. 
\begin{lemma}\label{finiteboundary}
	For each $\iota\in M$, when $N>N_0$, and for any $z>0$, we have that $b_\iota(z)<b^N_\iota(z)<\infty$. Moreover,   $$\PP_{\iota,x}(\tau_D<\infty)=1$$
	for all $(\iota,x)\in M\times(0,\infty)^2$.
\end{lemma}
\begin{proof}
Since  $b_\iota$ is decreasing, it suffices to prove that $b^N_\iota(0^+)<\infty$. 
By Proposition \ref{1-dtime}, for $N>N_0$,
	we have that  $\wt V^{1,N}(\iota,z,0)=0$ for $z>\bar x_\iota^{1}$. Therefore, for $z>\bar x_\iota^{1}$,
	$$\ba{ll}\ad V^N(\iota,z,0)=\inf_{\tau}\EE_{\iota,z,0} \int_{0}^{\t} e^{-\int_0^t\lambda(\a_s)ds}H^N(\a_t,(X^1_t,X^2_t))dt\\
	\ns\ad
	\geq \inf_{\tau}\EE_{\iota,z}\int_{0}^{\tau}e^{-\int_0^t\lambda(\a_s) ds}H^N(\a_t,(\wt X^1_t,0))dt =\wt V^{1,N}(\iota,z)=0.
	\ea$$
		 Thus 
		 $$b^N_\iota(z)\leq b^N_\iota(0^+)\leq \bar x_{\iota}^{1}.$$ Let us now define the two stopping times 
	\bea \tau_*:=\inf\{t\geq0:X^1_t\geq \sup_{\iota\in M}\bar x_\iota^{1}\}\text{ and } \tilde \tau_*:=\inf\{t\geq0:\wt X^1_t\geq \sup_{\iota\in M}\bar x_\iota^{1}\}.\eea
	Clearly, if $X^1_0=\wt X_0^1$, the following inequality for the first entry time 
	$$
	\tau_{D^N}\leq \tau_*\leq \tilde \tau^*,
	$$
	holds with probability 1. 
	By Assumption \ref{A3}, $\tilde \tau^*<\infty$ with probability 1, completing the proof.
	\end{proof}

We now present some further properties of the value function and of the optimal stopping boundary.
\begin{proposition}\label{region} For each $\iota\in M$, the following four statements hold:
	
	
	{\rm (1)} $b_\iota(\cdot)$ is decreasing, convex and continuous on $(0,\infty)$, with $b_\iota(0^+)<\infty$.\ss
	
	{\rm (2)} For any $z\in[0,\infty),\, $$H(\iota,(z,b_\iota(z)))\geq 0$ .\ss
	
	{\rm (3)}	If $(x_1,x_2)\in D_\iota$, and if $\bar x_1\geq x_1$ and $\bar x_2\geq x_2$, then $(\bar x_1,\bar x_2)\in D$. \ss

	{\rm (4)} For $z\geq \bar x_\iota^{2}$, $b_\iota(z)=0$.

\end{proposition}
\begin{proof}
	(1) Since $ V(\iota, \cdot)$ is concave,
	$b_\iota(\cdot)$ is convex.  By monotonicity of $ V(\iota,\cdot)$, $b_\iota(\cdot)$ is decreasing. Since convex functions are continuous on open sets, $b_\iota(\cdot)$ is  continuous on $(0,\infty)$. The boundedness of $b_\iota$ follows from the proof of Lemma \ref{finitetau}.
	
	(2) For any $x $ such that $H(\iota,x)\leq 0$, it is clear that instantaneous stopping is not optimal; this is because, for a given $x$, we have that $ V(\iota,x)<0$. Therefore, for any $z\in[0,\infty)$, $H(\iota,z,b_\iota(z))\geq 0$.
	
	(3) This immediately follows from the monotonicity of $ V(\iota,\cdot)$.
	
	(4) This immediately follows from Lemma \ref{finiteboundary}.
\end{proof}

We now present this section's key result regarding the optimal stopping and continuation regions. 
\begin{theorem}\label{OPSStop} The set $C$ is the optimal continuation region and $D$ is the optimal stopping region for the optimal stopping problem in \eqref{value}. Moreover, $\tau_D$ as given in \eqref{tauD} is an optimal stopping time for the optimal stopping problem in \eqref{value}, i.e.
	\bel{repD}V(\iota,x)=\EE_{\iota,x}\int_0^{\tau_D} e^{-\int_0^t\lambda(\a_s)ds}H(\a_t,X_t)dt.\eel 
\end{theorem}

\begin{proof}
By Proposition \ref{region}, $C^{N_0}$ is a bounded region for some $N_0>0$. Then \eqref{optrule} immediately holds since $H(\a_t,X_t)$ is bounded by the first entry time $\tau_{D^{N_0}}$ of $D^{N_0}$. The desired result then holds by Lemma \ref{thmops}.
	\end{proof}

\subsection{Regularity of $V$}\label{sec34}

The aim of this section is to prove the regularity of $ V$ from the representation given in \eqref{repD}. As we mentioned in the introductory section, since the continuation regions in different layers will be different, there exist so-called  {\it transition regions} on each layer.  For example, consider the set $C_1\cap D_2$, where we recall the definitions for $C_\iota$ and $D_\iota$ given in \eqref{contr-i}. Assuming that it is not empty, it is a transition region on the first layer. This is because when $(\a_t,X_t)\in\{1\}\times \big(C_1\cap D_2
\big)$, the process may jump into the stopping region on the second layer, i.e. in a sufficiently small time interval such that $X_t$ does not exit $D_2$). 
For the sake of completeness, we pause to emphasize that the regularity results in \cite{Ei1990} concerning a Dirichlet problem with the same boundaries in different layers cannot be applied in the present setting, as in our problem the boundaries are not the same in different layers.\\ 
\indent We proceed with some necessary notation. Let $\sC^k(U)$ be the set of $k$-th continuously differentiable functions in the domain $U$. Let the Sobolev space $\sW^{k,p}(U)$ be the set of functions with weak derivatives up to order $k$ having finite $L^p$ norm. When $k=2$ and $p=\infty$, $f\in \sW^{2,\infty}((1/K,K)^2)$ for any $K>1$ 
if and only if $Df$ is locally Lipschitz in  $ (0,\infty)^2$. We now turn to Lemma \ref{reprenstationhatV} below.

\begin{lemma} \label{reprenstationhatV}
	For each $\iota\in M$, $ V(\iota,\cdot)\in \sC^2\big(C_{\iota}\big)$ and $V$ solves the following partial differential equation in the `classical sense' (i.e.  $V$ is second order differentiable and the derivatives satisfy \eqref{dirichilet}; for reference, see Chapter 2 and Chapter 4 in \cite{GT}):
	\bel{dirichilet}\ba{ll}\ad\LL  V(\iota,x)-\lambda(\iota) V(\iota,x)=-H(\iota,x)\text{ for  }(\iota,x)\in  C.\ea\eel
\end{lemma}
\begin{proof}
Recall that the value function admits the following representation
$$  V(\iota,x)=\EE_{\iota,x}\(\int_0^{\t_D}e^{-\int_0^t\lambda(\a_s)ds}H(\a_t,X_t)dt\).$$
Let $\tau_1$ be defined as the first jump time of $\a$ and let $\tau_2$ be given by $\tau_2=\inf\{t\geq0:X_t\in D_\iota\}.$ Under the measure $\PP_{\iota,x}$, we immediately have that $\tau_1\wedge \tau_2\leq \tau_D$. Employing the strong Markov property of $(\a,X)$, it follows that (cf. (2.9) in \cite{Yin2009})
\begin{equation}\label{MM} V(\iota,x)=\EE_{\iota,x}\(\int_0^{\tau_1\wedge \tau_2}e^{-\int_0^t\lambda(\a_s)ds}H(\a_t,X_t)dt+e^{-\int_0^{\tau_1\wedge \tau_2}\lambda(\a_s)ds}V(\a_{\tau_1\wedge \tau_2},X_{\tau_1\wedge \tau_2})\)=:M_1+M_2.
\end{equation}
\indent We now turn our investigation to the regularity of $ V(\iota,\cdot)$ inside  $C_\iota$. Recall that $\tau_1$ is exponentially distributed with parameter $-q_{\iota\iota}$. Continuing from \eqref{MM} above, we calculate 
\bel{M10}\ba{ll}M_1\ad=\EE_{\iota,x}\int_0^\infty\int_0^{ \tau_2}-q_{\iota\iota}e^{q_{\iota\iota}s}e^{-\lambda(\iota)t}H(\a_t,X_t)I(0\leq t\leq s)dtds\\
\ns\ad=\EE_{\iota,x}\int_0^{ \tilde\tau_2}e^{(q_{\iota\iota}-\lambda(\iota))t}H(\iota,(\wt X^{\iota,1}_{ \tilde\tau_2},\wt X^{\iota,2}_{ \tilde\tau_2}))dt,\ea\eel
where $\wt X=(\wt X^{1,\iota},\wt X^{2,\iota})$ is the solution of the following system of stochastic differential equations
$$d\wt X_t^{\iota,i}=\(a^i(\iota)+\sum_{j=1}^2b^{ij}(\iota)\wt X_t^{\iota,j}\)dt+\sigma^i(\iota)\wt X_t^{\iota,i}dW_t^i\,\text{ for }i=1,2,$$
and $\tilde \tau_2=\inf\{t\geq 0:(\wt X^{\iota,1}_{ \tilde\tau_2},\wt X^{\iota,2}_{ \tilde\tau_2})\notin D_\iota\}$.
We also obtain that
\bel{M20}\ba{ll}M_2\ad=\EE_{\iota,x} e^{-\int_0^{\tau_1}\lambda(\a_s)ds}V(\a_{\tau_1},X_{\tau_1})I(\tau_2\geq\tau_1)+\EE_{\iota,x}e^{-\int_0^{\tau_2}\lambda(\a_s)ds} V(\a_{ \tau_2},X_{ \tau_2})I(\tau_2<\tau_1)\\
\ns\ad=\EE_{\iota,x}\int_0^\infty\sum_{\jmath\neq\iota}q_{\iota\jmath}e^{(q_{\iota\iota}-\lambda(\imath))s} V(\jmath,\wt X^{\iota,1}_{s},\wt X^{\iota,1}_{s})I(\tilde\tau_2\geq s)ds+\EE_{\iota,x}\( e^{(q_{\iota\iota}-\lambda(\iota))\t_2}V(\iota, X_{\t_2})\)\\
\ns\ad=\EE_{\iota,x}\int_0^{\tilde\tau_2}\sum_{\jmath\neq\iota}q_{\iota\jmath}e^{(q_{\iota\iota}-\lambda(\imath)) s} V(\jmath,\wt X^{\iota,1}_{s},\wt X^{\iota,2}_{s})ds+\EE_{\iota,x}\( e^{(q_{\iota\iota}-\lambda(\iota))\t_2}V(\iota,\wt X^{\iota,1}_{ \tilde\tau_2},\wt X^{\iota,2}_{ \tilde\tau_2})\).\ea\eel
It then follows from \eqref{M10} and \eqref{M20} that for any $x\in C_\iota$,
\bel{pppp}\ba{ll} V(\iota,x)\ad=\EE_{\iota,x}\int_0^{ \tilde\tau_2}e^{(q_{\iota\iota}-\lambda(\iota))t}\(H(\iota,(\wt X^{\iota,1}_{ \tilde\tau_2},\wt X^{\iota,2}_{ \tilde\tau_2}))+\sum_{\jmath\neq\iota}q_{\iota\jmath}  V(\jmath,\wt X^{\iota,1}_{ \tilde\tau_2},\wt X^{\iota,2}_{ \tilde\tau_2})\)dt\\
\ns\ad\q+\EE_{\iota,x}\( e^{(q_{\iota\iota}-\lambda(\iota))\tilde\t_2}V(\iota,\wt X^{\iota,1}_{ \tilde\tau_2},\wt X^{\iota,2}_{ \tilde\tau_2})\).\ea\eel
Invoking the ellipticity of the infinitesimal generator of $(\wt X^{\iota,1},\wt X^{\iota,2})$ and the local Lipschitz continuity of $ V(\jmath,\cdot)$  and $H(\jmath,\cdot)$ (which holds by concavity),  by \cite[Lemma 6.17]{GT}, we have that $  V(\iota,\cdot)\in\sW^{3,\infty}(C_\iota)$. That is,
$D^2  V (\iota,\cdot)$ is locally Lipschitz in $C_\iota$. Moreover,
$  V$ solves \eqref{dirichilet} in the classical sense. This concludes the proof.
\end{proof}

\begin{remark}\rm
Given $\tilde \LL $ is an elliptic operator and that $f$ is only assumed to be continuous, it is well known  that the solution $u$ to the elliptic partial differential equation $\tilde \LL u(x)=f(x)$ 
may not be twice continuously differentiable. This means that if $ V$ is only assumed  continuous, classical ellipticity theory does not guarantee that from \eqref{pppp} we will have that $  V(\iota,\cdot)\in \sC^2(C_\iota)$. Thus highlights the critical importance of the concavity of $  V(\iota,\cdot)$ in this setting.
\end{remark}

We now verify in Lemma \ref{smoothfit} below that the principle of smooth fit holds for the value function at the optimal stopping boundary. That is, we prove that for each $\iota\in M$, $V(\iota,\cdot)\in\sC^1\big((0,\infty)^2)$. 

\begin{lemma}\label{smoothfit}
	For each $\iota\in M$, $V(\iota,\cdot)\in\sC^1\big((0,\infty)^2\big)$.
	
\end{lemma}
\begin{proof} First, we see that the optimal stopping boundary  is {\it probabilistically regular} (see \cite{DePe})  by noting that, for $x_2=b_{\iota}(x_1)$,
	\bel{Hit}\ba{ll}\ad\liminf_{t\rightarrow 0^+}\PP_{\iota,x}\big((\alpha_t, X_t)\in D\big)\geq \liminf_{t\rightarrow 0^+}\PP_{\iota,x}\Big((\alpha_t, X_t)\in D, \text{ $\a$ does not jump between } [0,t]  \)\\
	\ns\ad=\liminf_{t\rightarrow 0^+}e^{q_{\iota\iota}t}\PP\( \wt X_t^{\iota,1}>x_1,\wt X_t^{\iota,2}>x_2\)=\frac14>0,\ea\eel 
	where in the second line we have employed statement (3) in Proposition \ref{region}. By Blumenthal's 0-1 law, $$\liminf_{t\rightarrow 0^+}\PP_{\iota,x}\big((\alpha_t, X_t)\in D\big)=1.$$ It is evident that $(\a,X)$ is a strong Feller process (\cite{Yin2009}). The optimal stopping boundary is Green's regular (\cite{KA1998,Pe-3}) in the sense that 
	\bel{tauD0}\tau_D^{\iota,x}\rightarrow 0\text{ as $x\in C\rightarrow \bar x\in S_{b_\iota}$.}\eel 
	\indent To complete the proof, we need to invoke the principle of smooth fit. Arguments (a) and (b) below prove that the principle of smooth fit holds in the present setting. \ss
	
	(a) Since $V(\iota,x)=0$ for $x_2\geq b_\iota(x_1)$ and $ V(\iota,x)$ is increasing, it follows that
	\bel{<0}\lim_{h\downarrow0}\frac{ V(\iota,x_1,b_{\iota}(x_2)+h)- V(\iota,x_1,b_{\iota}(x_2))}{h}=0\leq\lim_{h\downarrow0}\frac{ V(\iota,x_1,b_{\iota}(x_2)-h)- V(\iota,x_1,b_{\iota}(x_2))}{-h}.\eel
	
	(b) Since the boundary is Green's regular, note that by \eqref{tauD0} that as $h\rightarrow 0^+$, $\tau^{\iota,x_1,b_{\iota}(x_1)-h}_D\rightarrow 0$. It then follows that
	\bel{>0}\ba{ll}\ad\lim_{h\downarrow0}\frac{ V(\iota,x_1,b_{\iota}(x_1)-h)- V(\iota,x_1,b_{\iota}(x_1))}{-h}\\
	\ns\ad=\lim_{h\downarrow0}\frac1{-h}\EE\int_0^{\tau^{\iota,x_1,b_{\iota}(x_1)-h}_D}e^{-\int_0^t\lambda(\a_s)ds }H(\a_t,X_t(\iota,x_1,b_{\iota}(x_1)-h))dt\\
	\ns\ad\leq \lim_{h\downarrow0}\EE\[\int_0^{\tau^{\iota,x_1,b_{\iota}(x_1)-h}_D}e^{-\int_0^t\lambda(\a_s)ds} \frac{H(\a_t,X_t(\iota,x_1,b_{\iota}(x_1)))-H(\a_t,X_t(\iota,x_1,b_{\iota}(x_1)-h))}{h}dt\]\\
	\ns\ad\leq 0,\ea\eel
where in the second to last line we have used the fact that $H(\iota,\cdot)$ is locally Lipschitz and that as $h\rightarrow 0^+$, $\tau^{\iota,x_1,b_{\iota}(x_1)-h}_D\rightarrow 0$. From \eqref{<0} and \eqref{>0}, we see that $ V(\iota,\cdot)$ is differentiable  at the stopping boundary $S_{b_\iota}$. Since  $ V(\iota,\cdot)$ is concave, we have that $ V(\iota,\cdot)\in\sC^1\big((0,\infty)^2\big)$ since concave differentiable  functions are continuously differentiable on open sets (cf. \cite[Theorem 2.2.2]{Bor2010}). This concludes the proof.
\end{proof}
\begin{remark}\label{rempr}
	\rm One may similarly prove that both of  the boundaries of $D_\e^N$  and $\cup_ND_\e^N$  are probabilistically regular by invoking monotonicity of the boundaries. 
\end{remark}

Given that the principle of smooth fit holds for $V$, we can derive a refined regularity of $V(\iota,\cdot)$ on the whole coordinate plane. We do so in Proposition \ref{refineregularity} below.
\begin{proposition}\label{refineregularity}
	For any $K>1$, we have that	$ V(\iota,\cdot)\in\sW^{2,\infty}\big((1/K,K)^2\big)$.
	
\end{proposition}

\begin{proof}
	Note that $ V(\iota,\cdot)\in\sC^2\big(C_\iota\big)\cap\sC^1\big((0,\infty)^2\big)$. Since $V(\iota,\cdot)$ is concave, we have on $C_\iota$ that $$\begin{pmatrix}\partial_{11}  V&\partial_{12}  V\\
	\partial_{12}  V&\partial_{22}  V\end{pmatrix}\leq 0.$$
	Therefore, $\partial_{11}  V \leq 0$ and	$\partial_{22}  V\leq 0$, and
	\bel{det} \partial_{11}  V\partial_{22}  V-(\partial_{12}  V)^2\geq 0.\eel
	Further, on $C_\iota$, 
	\bel{w2}\ba{ll}\ad -H(\iota,x)=\LL  V(\iota,x)-\lambda(\iota) V(\iota, x)\\
	\ns\ad=\sum_{\jmath=1}^2q_{\iota\jmath}  V(\jmath,x) -\lambda(\iota) V(\iota,x)\\		\ns\ad\qq+\sum_{i=1}^2\(a^i(\iota)

	+\sum_{j=1}^2b^{ij}(\iota)x_j\)\partial_i V(\iota, x)+\frac12\sum_{i=1}^2(\si^i(\iota)x_i)^2\partial_{ii}V(\iota, x).
	\ea\eel
When		 $x$ converges to some point on the optimal stopping boundary in $(1/K,K)^2$, both $\partial_{11}  V$ and $\partial_{22}  V$ are bounded in $(1/K,K)^2\cap C_\iota$. 
		By \eqref{det}, $\partial_{12}  V$ is bounded in $(1/K,K)^2\cap C_\iota$. Combining this with the fact that $  V(\iota,\cdot)\in \sC^1\big((0,\infty)^2\big)$ yields that $\partial_1 V(\iota,\cdot)$ and $\partial_2 V(\iota,\cdot)$ are uniformly Lipschitz in $(1/K,K)^2$, i.e. $  V(\iota,\cdot)\in\sW^{2,\infty}\big((1/K,K)^2\big)$. 
\end{proof}

\section{Free-boundary problem}\label{sec:fbp}
Provided that $\partial [0,\infty)^2$ is the natural boundary for $X$, we expect $V$ to be the solution to the following free-boundary problem
\bel{freeboundayPDE}\left\{\ba{ll}
\ad\LL w(\iota,x)- \lambda(\iota) w(\iota,x) = -H(\iota,x)\;\;
\text{for}\;\; (\iota,x) \in U; \\
\ns\ad w= 0 \text{ on }U^c;\q\partial_1 w=\partial_2 w = 0\text{ on }\partial U.
\ea\right.\eel

We aim to solve the above free-boundary problem in \eqref{freeboundayPDE} by finding a couple  $(w, U)$  
such that $ w(\iota,\cdot)\in  \sC^1((0,\infty)^2)\cap\sC^2(U)$ and such that \eqref{freeboundayPDE} holds in the classical sense  as described as Lemma \ref{reprenstationhatV} above. Theorem \ref{thm4} below connects the value function of our optimal stopping problem in \eqref{value}
with the solution of the free-boundary problem. It is important to note that the first boundary condition in the second line of \eqref{freeboundayPDE} \textit{cannot} be replaced by `$w=0$ on $\partial U$' because  of the aforementioned transition regions associated with the discrete states of the Markov chain.

\begin{theorem} \label{thm4} As before, we suppose that  Assumptions \ref{A1}, \ref{A2} and \ref{A3} hold.  The value function $V$ together with the optimal continuation region $C$ is the classical solution of the free-boundary problem in \eqref{freeboundayPDE}. Restricting $\partial C$  to be in the `admissible class' {\rm(}the definition to be given in Theorem \ref{mainthm}{\rm)} of the free-boundary problem, the solution of \eqref{freeboundayPDE} is unique. 
	
\end{theorem}
\begin{proof}
	(1)  By  Lemma \ref{reprenstationhatV} and Lemma \ref{smoothfit}, $ V$  verifies the differential equation in the classical sense on the region $C$. It also satisfies the boundary conditions in \eqref{freeboundayPDE}.
	
	(2) In the next section, we will see that the optimal stopping boundary $\partial C$ is unique if restricted to the `admissible' class. The reason is that the optimal stopping boundary solves a nonlinear integral equation which has a unique solution in the `admissible' class. Provided that the stopping boundary $\partial C$ is unique, uniqueness  holds by applying It\^o's formula. This completes the proof.\end{proof}

\section{Nonlinear integral equation}\label{sec:nie}
In this section, we will show that the optimal stopping boundaries  can be characterized as a unique solution to a coupled system of nonlinear integral equations.  We first recall that $(\a,X)$ is a homogeneous strong Feller process if the initial value is in $ M\times(0,\infty)^2$. Let $p_t(\iota,x;\jmath,y)$ be its transition density, i.e.
$$\PP_{\iota,x}(\a_t=\jmath,X_t\in dy)=p_t(\iota,x;\jmath,y)dy.$$
%
We now present the main theorem of this paper. 
\begin{theorem}\label{mainthm} Suppose $\lambda(\cdot)=\lambda_0>0$ is a constant and that Assumptions \ref{A1}, \ref{A2} and \ref{A3} hold. The optimal stopping boundary $\{b_{\iota}\}_{\iota\in M}$ is the unique solution to the following nonlinear integral equation
	\bel{bouneq}\int_0^\infty\sum_{\jmath\in M}\int_0^\infty\int_0^{b_\jmath(y_1)} e^{-\l_0t}p_t(\iota,x_1,b_\iota(x_1); \jmath,y_1,y_2)H(\jmath,y) dy_2dy_1 dt=0\eel
	for any $x_1\in(0,\infty)$ in the `admissible' class, where the `admissible' class consists of all possible bounded continuous functions such that $H(\iota,z,b_\iota(z))\geq0$ for any $(\iota,z)\in M\times(0,\infty)$. Moreover the value function has the following representation 
	\bel{reV} V(\iota,x)=\int_0^\infty\sum_{\jmath\in M}\int_0^\infty\int_0^{b_\jmath(y_1)}e^{-\l_0t}p_t(\iota,x; \jmath,y)H(\jmath,y) dy_2dy_1 dt.\eel
	
\end{theorem}

\begin{proof} We first recall the necessary results on regularity for $V$ in Section \ref{sec34}. For $K>1$, we have that
	\bel{vregu}V(\iota,\cdot)\in\sC^2(C_\iota)\cap\sC^1\big((-\infty,\infty)^2\big)\cap \sW^{2,\infty}\big((1/K,K)^2\big).\eel
	Let the function `$\text{Dis}$' be defined as
	$$\text{Dis}(x, S):=\inf_{\bar x\in S}|x-\bar x|.$$
	\noindent For a given $\d>0$, let
	\bel{defCDe}\left\{\ba{ll} 
	\ad C_\iota^{K,\d}:=\Big\{x\in(1/K,K)^2\cap C_\iota:\text{Dis}(x, \partial C_\iota)>\d\Big\} \\
	\ns\ad D_\iota^{K,\d}:=\Big\{x\in(1/K,K)^2\cap D_\iota:\text{Dis}(x, \partial C_\iota)>\d\Big\}\\
	\ns\ad E_{\iota}^{K,\d}:=(1/K,K)^2\cap(C_\iota^{K,\d}\cup D_\iota^{K,\d})^c.\ea\right.\eel

\noindent For sufficiently small $\e>0$, let $J_\e(x)=\e^{-2}J(x/\e)$  be a mollifier, i.e. a non-negative function $J\in\sC^\infty\big((-\infty,\infty)^2 \big)$ supported on $(-1,1)^2$ with $\int_{-\infty}^\infty\int_{-\infty}^\infty J(x)dx=1$. Further, for $x\in(1/K,K)^2$, let
	$$  V^\e(\iota,x):=\int_{\dbR^2}J_\e(y)  V(\iota,x-y)dy.$$
	 We shall use the notation $*$ to represent the convolution operator. Since $  V(\iota,\cdot)\in\sC^1\big((-\infty,\infty)^2\big)$, $D  V^\e(\iota,x)=J_\e*D  V(\iota,x)$ for $x\in (1/K,K)^2$ when $\e<1/K$. We also know that $  V^\e(\iota,\cdot)\in\sC^2(C_\iota)$ and that for $x\in C_\iota^{K,\d}$ the following equality holds when $\e<1/K$ and $\e<\d$
	  $$D^2  V^\e(\iota,x)=J_\e*D^2  V(\iota,x).$$   

We now document the following four facts for $  V^\e$ when both $\e<1/K$ and $\e<\d$. All of these facts easily follow from the $\sC^\infty$ properties of mollifiers. 
	\begin{itemize}
		\item[(Fact1)]  Since $  V(\iota,\cdot)\in \sC^1\big((0,\infty)^2\big)$, $  V^\e(\iota,\cdot)$ is $\sC^1$ in $(1/K,K)^2$ up to the  boundary, i.e. $  V^\e(\iota,\cdot)\in\sC^1\big([1/K,K]^2\big)$.  Moreover, $  V^\e(\iota,\cdot)$ converges to $  V(\iota,\cdot)$ uniformly  on $[1/K,K]^2$. 
		\item[(Fact2)] Since $J\in\sC^\infty\big((0,\infty)^2\big)$, $  V^\e(\iota,\cdot)\in \sC^\infty\big((1/K,K)^2\big)$.
		\item[(Fact3)] Since, for any $K>1$,  $  V(\iota,\cdot)\in\sW^{2,\infty}\big((1/K,K)^2\big)$, the second order derivatives of $  V^\e(\iota,\cdot) $ are bounded on $(1/K,K)^2$ by a constant independent of $\e$ for $\e$ in $(0,\delta)$.
		This can be seen from the fact that  the Lipschitz constant of $D  V^\e(\iota,\cdot)$ is bounded by the Lipschitz constant of $D  V(\iota,\cdot)$ on $(1/K,K)^2$.
		\item [(Fact4)]   For any $\d>0$, since $  V(\iota,\cdot)=0$ on $D_\iota$, for $\e$ sufficiently small, $  V^\e(\iota,\cdot)=0$ on $ D_\iota^{K,\d}$. Note that $  V(\iota,\cdot)\in \sC^2(C_\iota)$ and $\partial_{11}   V(\iota,\cdot),\partial_{22}   V(\iota,\cdot)$ are bounded on $C_{\iota}^{K,\d}$ (a bounded subset of $C_\iota$ with compact closure). Therefore, $\partial_{11}  V^\e(\iota,\cdot)$ and $\partial_{22}  V^\e(\iota,\cdot)$ converge uniformly on $C_\iota^{K,\d}$ to $\partial_{11}   V(\iota,\cdot)$ and $\partial_{22}  V(\iota,\cdot)$, respectively. 
	\end{itemize} 
	Let us now define $$\rho^t_K=t\wedge\inf\{s\geq 0:(\a_s,X_s)\notin M\times(1/K,K)^2\}.$$ By  (Fact1) and (Fact2), and employing It\^o's formula for $  V^\e(\a_t,X_t)$, we have that
	\bel{decM} \ba{ll}\ad \EE_{\iota,x}    e^{-\l_0 \rho^t_K}V^\e(\a_{\rho^t_K},X_{\rho^t_K})-  V^\e(\iota,x)\\
	\ns\ad =\EE_{\iota,x}\int_0^{\rho^t_K}e^{-\l_0s}\(\LL   V^\e(\a_s,X_s)-\lambda_0V^\e(\a_s,X_s)\)ds\\
\ns\ad=\EE_{\iota,x}\int_0^{\rho^t_K}e^{-\l_0s}\(\LL  V^\e(\a_s,X_s)-\lambda_0  V^\e(\a_s,X_s)\\
\ns\ad\qq\qq\qq\qq-\big[\LL  V(\a_s,X_s)-\lambda _0 V(\a_s,X_s)\big]\)I\big(X_s\in C_{\a_s}^{K,\d}\big)ds\\
\ns\ad\q+\EE_{\iota,x}\int_0^{\rho^t_K}e^{-\l_0s}[\LL   V^\e(\a_s,X_s)-\lambda _0  V^\e(\a_s,X_s)]I\big((X_s)\in E_{\a_s}^{K,\d}\big)ds\\
\ns\ad\q+\EE_{\iota,x}\int_0^{\rho^t_K}e^{-\l_0s}[\LL   V(\a_s,X_s)-\lambda  _0 V(\a_s,X_s)]\[I\big(X_s\in C_{\a_s}^{K,\d}\big)-I\big(X_s\in C_{\a_s}^{K}\big)\]ds\\
\ns\ad\q+\EE_{\iota,x}\int_0^{\rho^t_K}e^{-\l_0s}[\LL   V(\a_s,X_s)-\lambda _0  V(\a_s,X_s)]I\big(X_s\in C_{\a_s}^{K}\big)ds\\
\ns\ad=:M^{K,\d,\e,}_1+M^{K,\d,\e}_2+M^{K,\d,\e}_3+M^{K}_4.\ea\eel
	Sending $\e\rightarrow 0^+$, it follows by (Fact4) that 
	\bel{M1}\lim_{\e\rightarrow 0^+}|M_1^{K,\d,\e}|=0.\eel
	By (Fact3), for some $K>0$ independent of $\e$, 
	\bel{M2}\ba{ll}\ds\lim_{\d\rightarrow 0^+}\lim_{\e\rightarrow 0^+}|M_2^{K,\d,\e}|\ad\leq K\lim_{\d\rightarrow 0^+}{\EE_{\iota,x}}\int_0^{\rho_K^t}I\big(X_s\in E_{\a_s}^{K,\d}\big)ds\\
	\ns\ad \leq K\lim_{\d\rightarrow0^+}\int_0^{t}{\EE_{\iota,x}} I\big(X_s\in E_{\a_s}^{K,\d}\big)ds=0.\ea\eel	
	Similarly, it follows that 
	\bel{M3}\lim_{\d\rightarrow 0^+}\lim_{\e\rightarrow 0^+}|M_3^{K,\d,\e}|=0.\eel	
	Displays \eqref{decM}\textendash\eqref{M3}, together with (Fact1), enable us to conclude that
	$$ \ba{ll}\ad\EE_{\iota,x}    e^{-\l_0 \rho^t_K}V(\a_{\rho^t_K},X_{\rho^t_K})-  V(\iota,x)\\
	\ns\ad=\EE_{\iota,x}\int_0^{\rho_K^t}e^{-\l_0 s}[\LL   V(\a_s,X_s)-\lambda V(\a_s,X_s)]I\big(X_s\in C_{\a_s}^{K}\big)ds\\
	\ns\ad=-\EE_{\iota,x}\int_0^{\rho_K^t}e^{-\l_0 s} H(\a_s,X_s)I\big(X_s\in C_{\a_s}^{K}\big)ds.\ea$$
We now let $K\rightarrow\infty$ and $t\rightarrow\infty$. Since $V$ is bounded, dominated convergence yields that
	\bel{reVVVV} \ba{ll}\ad  V({\iota,x})=\EE_{\iota,x}\int_0^{\infty}e^{-\l_0 t}H(\a_t,X_t)I\big(X_t\in C_{\a_t}\big)dt.\ea\eel
	This is equivalent to \eqref{reV} by the definition of $C$ in \eqref{contr}.
	
	Note that $  V=0$ on $\partial C$; further, \eqref{bouneq} 	holds directly from the fact that $V=0$ on the optimal stopping boundary. To show that $\{b_{\iota}\}_{\iota\in M}$ is a unique solution to the
	equation  \eqref{bouneq} in the specified class of functions, one can
	adopt the four-step procedure from the proof of uniqueness given in
	\cite[Theorem 4.1]{DuPe}. 
	Given that the
	present setting creates no additional difficulties we will omit
	further details of this verification. This completes the proof.
\end{proof}

	

We now disclose the solution to the optimal stopping problem in \eqref{value}, which follows directly as a corollary of Theorem \ref{mainthm} above.
\begin{corollary} \label{cor51} Under Assumptions \ref{A1}, \ref{A2} and \ref{A3} and $\lambda(\cdot)=\lambda_0$,
 the optimal stopping time for the optimal stopping problem in \eqref{value}  is given by
	\bea\ba{ll}\ds\tau_*\ad=\inf\Big\{t\geq 0: X^2_t \geq b_{\a_t}(X^1_t)\Big\},\ea\eea
	where $\{b_\iota\}_{\iota\in M}$ is the unique solution to \eqref{bouneq} in the admissible class as defined  in Theorem \ref{mainthm}.
\end{corollary}

We conclude this section by offering a graphical illustration of Corollary \ref{cor51} in Figure 1 below. In this figure, the continuation region and stopping region have two layers and the blue line represents the optimal stopping boundary. The red line represents a possible path of $(\a,X)$ which, at a given time, jumps from the second layer to the first layer. The optimal stopping time $\tau_*$ is the first entry time of $(\a,X)$ into the region $D$ at some point on the optimal stopping boundary in the first layer. 


\begin{figure}[H]
	\centering
	\includegraphics[width=0.7\textwidth]{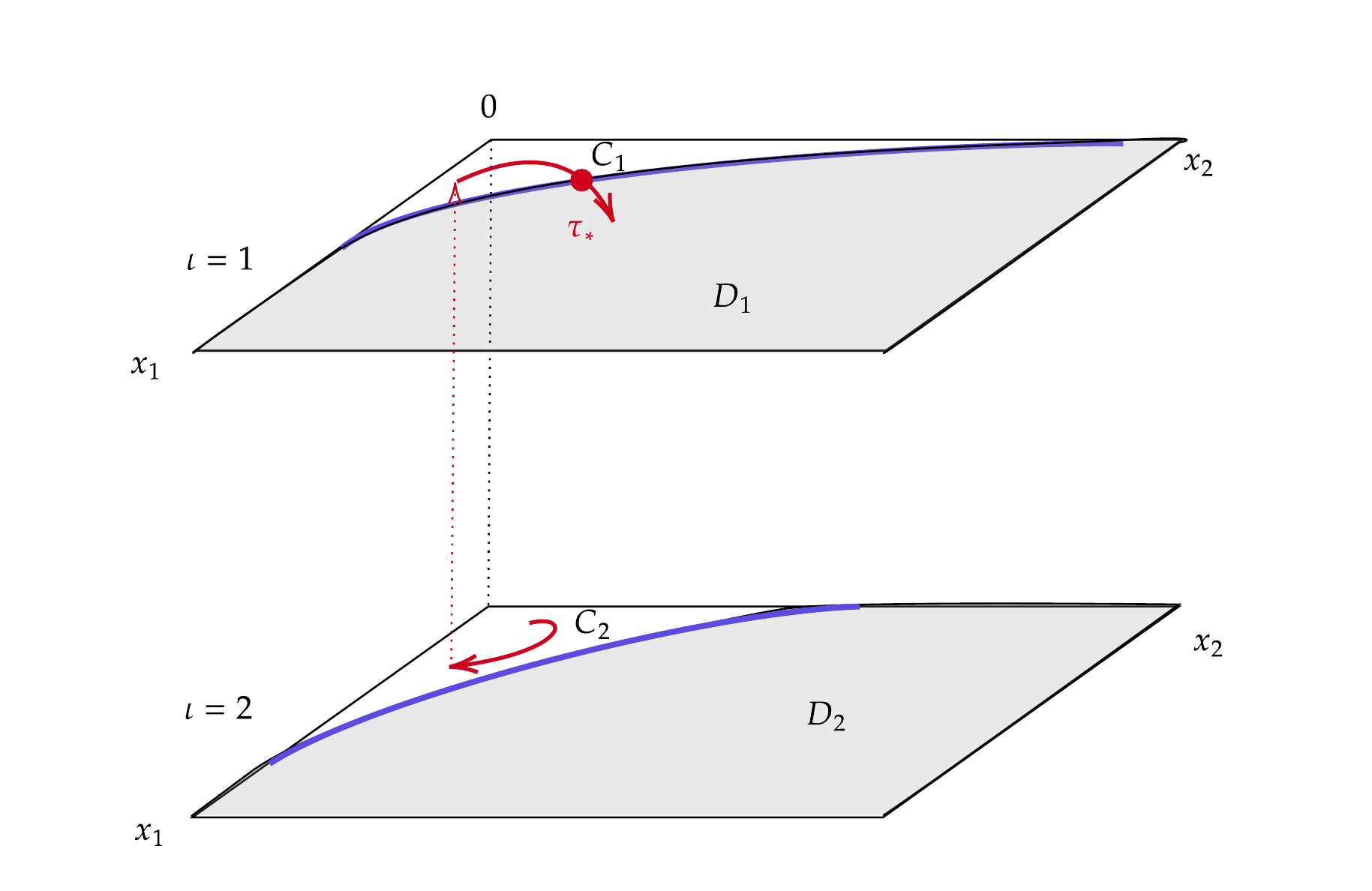}
	\caption{The optimal stopping time $\tau_*$ is the first entry time of $(\a,X)$ into $D$ at some point on the optimal stopping boundary in the first layer.} 
\end{figure}

\section{Quickest real-time detection of a Markovian drift}\label{sec:exp}

In this section, we apply this paper's results to a quickest detection problem of a Brownian coordinate drift that was solved in the one-dimensional case in \cite{GS} and solved in the multidimensional case in \cite{Er2020}, but now with random switching environments incorporated. 
Let $\a=\{\a_t:t\geq0\}$ be a continuous time Markov chain with finite state space $M=\{1,\!\cdots\!,n\}$ and generator $Q=(q_{\iota\jmath})_{n\times n}$, i.e. for all $\iota\neq\jmath$, $\sum_{\jmath\in M}q_{\iota\jmath}=0$ and  $ q_{\iota\jmath}\geq0$. As in \cite{Er2020,GS}, we consider a Bayesian formulation of the quickest detection problem. That is, we assume that one observes a sample path of the standard
two-dimensional Brownian motion $X=(X^1,X^2)$ with zero
drift initially, and then at some random and unobservable time $\theta>0$
taking value $0$ with probability $\pi \in [0,1)$ and being
exponentially distributed with parameter $\lambda>0$, one of the coordinate  processes $X$ obtains a
(known) non-zero drift $\mu$ permanently 
depending on the Markov chain.
The aim is to detect the time $\theta$  as `accurately' (to be specified below) as possible.

Based on the above formulation,  the observed process $X=(X^1,X^2)$ solves the following stochastic differential equations
\bel{SDE-e}
dX^i_t=\mu(\alpha_t)I(t\geq\theta,\beta=i)dt+dB^i_t\text{ for }i=1,2,
\eel
%
where $\beta$ denotes the number of the coordinate process which obtains the Markovian drift. We suppose that the prior distribution of $\beta$ is given with $\PP(\beta=i)=p_i\geq 0$ with $p_1+p_2=1$. Further, the unobservable time $\theta$, the random variable $\beta$, and the driving Brownian motion $B$ are all assumed to be independent.

%
%
%
%
%
%

Being based upon the continuous observation of $(\a,X)$, the problem is to find a
stopping time $\tau_*$ of $(\a,X)$, i.e. a stopping time with respect to the natural filtration $\cF^{\a,X}_t=\sigma(\a_s,X_s: 0\leq s\leq t)$ augmented with all $\PP$-null sets that is `as close as possible' to the unknown time $\th$. We formalize the phrase `as close as possible' by introducing the following cost functional for the stopping time $\tau$
\bel{cost-e} J(\tau)=\PP(\tau<\theta)+c\EE\[ F(\tau-\theta)I(\tau>\theta)\],\eel
where $F(t):=e^{\gamma t}-1$  for a given $\gamma>0$.  In the above equation, the first term corresponds to the probability of  {\it false alarm } and the second term corresponds to  {\it  expected exponentially penalized detection delay}. Therefore, the value function of our quickest detection problem is equivalent to the minimization problem
\bel{value-e} V=\inf_{\tau}J(\tau),\eel
where the infimum is taken for all (bounded) stopping times of $(\a,X)$.
%
\subsection{Measure change}
In this section, we will reconstruct the stochastic process in \eqref{SDE-e} and the minimization problem in \eqref{value-e} on  a new probability measure space where the quickest detection problem \eqref{value-e} can be reformulated as an optimal stopping problem through a suitable change of measure.

We begin by considering a probability measure space $(\Omega,\cF,\PP^0)$ supporting a two-dimensional Brownian motion $X$. This space shall also support a Markov Chain $\alpha$ with transition matrix $Q$ as given above, a random variable $\beta$ with $\PP^0(\beta=i)=p_i$, and a random variable $\th$ with $\PP^0(\th=0)=\pi$ and $\PP^0(\th>t)=(1-\pi)e^{-\l t}$ for $t>0$. Further, $X$, $\a$, $\beta$ and $\th$ are all assumed independent under $\PP^0$. Let the natural filtration of $(\a,X)$ be $\dbF$ and its augmentation by $\sigma(\th,\beta)$ 
be $\dbG$ (i.e. $\dbG=\{\cG_t\}_{t\geq 0}$ where $\cG_t=\sigma(\th,\beta,\{\a_s,X_s:s\leq t\})$. Our first task is to construct a new probability measure $\PP$ such that  the probability measure of $(\a,X)$ under $\PP$ coincides with the solution of \eqref{SDE-e}.

We now proceed towards this end. Let \bel{Zmeasure-e}Z_t:=\exp\Big\{\sum_{i=1}^2\(\int_0^t\mu(\alpha_t)I(s\geq\theta,\beta=i)dX^i_s-\int_0^t\mu^2(\alpha_t)I(s\geq\theta,\beta=i)ds\)\Big\}.\eel
We now define a new probability measure $\PP$ on $(\Omega,\cF)$ such that, for every $t\geq 0$,
$$\frac{d\PP}{d\PP^0}\Big|_{\cG_t}=Z_t,$$
and where $\cF=\sigma(\cG_t:t\geq 0)$. By  Girsanov's theorem, under the measure $\PP$,
$$X_t-\vec\mu(t;\a,\beta,\th)$$
is a two-dimensional standard Brownian motion,
where 
$$\vec\mu(t;\a,\beta,\th):=\begin{pmatrix}
\ds\int_0^t\mu(\a_s)I(s\geq\theta,\beta=1)ds\\\ns\ds\int_0^t\mu(\a_s)I(s\geq\theta,\beta=2)ds
\end{pmatrix}.$$
Since $Z_0=1$, $\PP$ and $\PP^0$ coincide on $\cG_0=\sigma(\beta,\th)$, i.e. the distributions of $\th$ and $\beta$ under $\PP^0$ and $\PP$ are the same. 
For the new measure $\PP$, given the initial value $\a_0=\iota$ and the constant $\pi$, the cost function $J$ can be represented as
\bel{JP0-e}\ba{ll}J(\t;\iota,\pi)\ad=\PP^0(\tau<\th)+c\EE^0\[ F(\tau-\theta)I(\tau>\theta)\]\\
\ns\ad=\sum_{i=1}^2p_i\(\PP^0_i(\tau<\th)+c\EE^0_i\big[ F(\tau-\theta)I(\tau>\theta)\big]\),\ea\eel
where  
$\PP_i$ is the probability measure under $\PP$ given $\beta=i$. We now reformulate the problem using  the measure  $\PP^0$.

To tackle \eqref{JP0-e}, we consider the posterior probability distribution process $\varPi_t=(\varPi^1_t,\varPi^2_t)$ and the corresponding  weighted likelihood ratio process $(\varPhi,\varPsi)$ given the data $(\a,X)$ observed up until time $t$, where the aforementioned quantities are defined as
\bel{Posterior}\left\{\ba{ll}\ad\varPi^i_t:=\PP^0_i(\theta\leq t|\cF_t^{\a,X}),\text{ for }i=1,2,\\
\ns\ad\varPhi_t:=\frac{\EE^0_1\big[e^{\gamma(t-\theta)}I(\theta\leq t)\big|\cF_t^{\a,X}\big]}{1-\varPi^1_t},\\
\ns\ad\varPsi_t:=\frac{\EE^0_2\big[e^{\gamma(t-\theta)}I(\theta\leq t)\big|\cF_t^{\a,X}\big]}{1-\varPi^2_t}\ea\right.\eel
where $\cF^{\a,X}$ is the natural filtration of $(\a,X)$ augmented with all $\PP^0$-null sets.
By Girsanov's theorem, 
$$\frac{d \PP_i}{d\PP_i^0}\Big|_{\cG_t}=Z^i_t,$$
where $(Z^1,Z^2)$ is defined by
\bel{lrprocess}Z^i_t:=\exp\Big\{\int_0^t\mu(\a_s)I(s\geq\th)dX^i_s-\frac12\int_0^t\mu^2(\a_s)I(s\geq\th)ds\Big\}
.\eel

We now wish to use $(\varPhi,\varPsi)$ to represent the cost functional $J$  under the measure $\PP$. 
Under the measure $\PP$, for $\tau<\th$,  $Z^i_\th=1$ almost surely.  It then follows that
\bel{cost-1-part}\PP^0_i(\tau<\theta)=\EE_i^0(Z^i_\theta I(\tau<\th))=\PP_i(\tau<\theta)=1-\pi-(1-\pi)\EE_i\int_0^\tau \lambda e^{-\lambda t}dt.\eel
Using the formula for conditional probability, we obtain
$$\ba{ll}\ad1-\varPi^i_t=\PP^0_i(\theta> t|\cF_t^{\a,X})=\frac{\EE_i [Z^i_tI(t<\theta)|\cF_t^{\a,X}]}{\EE_i [Z^i_t|\cF_t^{\a,X}]}=\frac{(1-\pi)e^{-\l t}}{\EE_i [Z^i_t|\cF_t^{\a,X}]},\ea$$
where in the second equality we have employed the assumption that $X$, $\a$, $\beta$ and $\th$ are independent under $\PP^0$ and recalled that $Z_t^i=1$ almost surely on $\{t\leq \th\}$. We continue by calculating
\bel{varPhi0}\ba{ll}\varPhi_t\ad=\frac{\EE^0_1\big[e^{\gamma(t-\theta)}I(\theta\leq t)\big|\cF_t^{\a,X}\big]}{1-\varPi^1_t}\\
\ns\ad=\frac{\EE_1\big[Z^1_te^{\gamma(t-\theta)}I(\theta\leq t)\big|\cF_t^{\a,X}\big]}{(1-\varPi^1_t)\EE_1 [Z^1_t|\cF_t^{\a,X}]}=\frac{\EE_1\big[Z^1_te^{\gamma(t-\theta)}I(\theta\leq t)\big|\cF_t^{\a,X}\big]}{(1-\pi)e^{-\lambda t}}.\ea\eel
Further, 
\bel{cost-2-part}\ba{ll}\ad\EE^0_1\[ F(\tau-\theta)I(\tau>\theta)\]=\g\EE^0_1\[I(\tau>\theta)\int_{\theta}^\t e^{\gamma (s-\theta)}ds\]\\
\ns\ad=\g\EE^0_1\[\int_{0}^\infty I(t\geq\theta) I(t<\tau)e^{\gamma (s-\theta)}dt\]\\
\ns\ad=\g\EE_1\[\int_{0}^\infty Z^1_tI(t\geq\theta) I(t<\tau)e^{\gamma (t-\theta)}dt\]\\
\ns\ad= \g\EE_1 \left\{\int_{0}^\infty I(t<\tau)\EE_1 \[Z^1_tI(t\geq\theta) e^{\gamma (t-\theta)}\Big|\cF_t^{\a,X}\]dt\right\}\\
\ns\ad=(1-\pi)\g\EE_1 \int_{0}^\t e^{-\lambda t}\varPhi_tdt.\ea\eel
Similarly, we may conclude that 
\bel{cost-3-part}\ba{ll}\ad\EE^0_2\[ F(\tau-\theta)I(\tau>\theta)\]=(1-\pi)\g\EE_2 \int_{0}^\t e^{-\lambda t}\varPsi_tdt.\ea\eel
Together with \eqref{cost-1-part}, \eqref{cost-2-part} and  \eqref{cost-3-part}, \eqref{cost-e} implies that 
\begin{equation}\label{newrepresentation}
J(\t;\iota,\pi)=1-\pi+c\gamma(1-\pi)\EE_{\iota,\frac\pi{1-\pi},\frac\pi{1-\pi}} \int_{0}^\t e^{-\lambda t}\(p_1\varPhi_t+p_2\varPsi_t-\frac\lambda{c\gamma }\)dt,
\end{equation}
where subscript under $\EE$ denotes the initial data of $(\a,\varPhi,\varPsi)$.\\
%
\indent We now derive the stochastic differential equation for $(\varPhi,\varPsi)$ under the measure $\PP$. We write the likelihood ratio process as
\bel{Lprocess}L^i_t:=\exp\Big\{\int_0^t\mu(\a_s)dX^i_s-\frac12\int_0^t\mu^2(\a_s)ds\Big\}
.\eel
It can now be easily seen  from \eqref{lrprocess} that $$Z_t^i=I(t< \th)+\frac{L^i_t}{L^i_\th}I(t\geq \th).$$
Invoking the assumption that $X$, $\a$, $\th$ and $\beta$ are independent under $\PP^0$,  we see from \eqref{varPhi0} that
\bel{varPhi}\varPhi_t=\frac{e^{\lambda t}}{(1-\pi)}\EE_1\big[Z^1_te^{\gamma(t-\theta)}I(\theta\leq t)\big|\cF_t^{\a,X}\big]=e^{(\lambda+\gamma)t}L^1_t\(\frac{\pi}{1-\pi}+\int_0^t\frac{\l e^{-(\lambda+\gamma)s}}{L^1_s}ds\).\eel
Similarly, we have that
\bel{varPsi}\varPsi_t=e^{(\lambda+\gamma)t}L^2_t\(\frac{\pi}{1-\pi}+\int_0^t\frac{\l e^{-(\lambda+\gamma)s}}{L^2_s}ds\).\eel
Under the measure $\PP$,  $\hat B_t=X_t$ is standard Brownian motion and
$dL^i_t=\mu(\alpha_t)L^i_tdX^i_t$. Therefore, under the measure $\PP$,  the couple $(\varPhi,\varPsi)$ 
solves the following system of stochastic differential equations with  Markov switching
\bel{sdephipsi}\left\{\ba{ll}\ad d\varPhi_t=[\lambda+(\lambda+\gamma)\varPhi_t]dt+\mu(\alpha_t)\varPhi_td\hat B^1_t\\
\ns\ad d\varPsi_t=[\lambda+(\lambda+\gamma)\varPsi_t]dt+\mu(\alpha_t)\varPsi_td \hat B^2_t.\ea\right.\eel
 From \eqref{lrprocess}, \eqref{varPhi}, and \eqref{varPsi}, $(\a,\varPhi,\varPsi)$ is observable in real time, and these are the sufficient statistics for our quickest detection problem.

We now offer an equivalent formulation of our quickest detection problem as the following optimal stopping problem
\bel{hatV}\hat V(\iota,\varphi,\psi)=\inf_\t\EE_{\iota,\varphi,\psi} \int_{0}^\t e^{-\lambda t}\(p_1\varPhi_t+p_2\varPsi_t-\frac\lambda{c\gamma }\)dt,\eel
where the infimum is taken over for all stopping times of $(\a,\varPhi,\varPsi)$ solving the system of equations in \eqref{sdephipsi}. The value function may then be represented as \bel{value-2} V(\iota,\pi)=(1-\pi)\parens{1+c\gamma\hat V\parens{\iota,\frac{\pi}{1-\pi},\frac{\pi}{1-\pi}}}.\eel

In order to apply the results in the present paper to this problem, we need only verify that Assumptions \ref{A1}, \ref{A2} and \ref{A3} from Section \ref{sec:pre} hold for the optimal stopping problem in \eqref{hatV}.
Indeed, in the present problem, we have that $a^i(\iota)=\l$, $b^{ii}(\iota)=\l+\g$, $b^{ij}(\iota)=0$ if $i\neq j$, $\sigma^i(\iota)=\mu(\iota)$, $\lambda(\iota)=\lambda$ and  $$H(\iota,\varphi,\psi)=p_1\varphi+p_2\psi-\frac\l{c\g}.$$
This verifies Assumptions \ref{A1}, \ref{A2}, and \ref{A3}.
%
Combining Theorem \ref{OPSStop} and Theorem \ref{mainthm},  we disclose in Corollary \ref{cor61} below the key result for the real-time quickest detection problem with switching states.

\begin{corollary} \label{cor61}
	Given $\pi\in[0,1)$, $\lambda,\g,\mu(\cdot)>0$ and  $p_1,p_2\geq0$ with $p_1+p_2
	\!=\!1$, the  quickest detection problem given in \eqref{value-e} admits the following representation
	\bea V(\iota,\pi)=(1-\pi)\[1+c\g\hat V\Big(\iota,\frac \pi{1-\pi},\frac \pi{1-\pi}\Big)\],\eea 
	where $\hat V$ is given by \eqref{hatV} above. The optimal stopping time is given by
	\bea\ba{ll}\ds\tau^*\ad=\inf\Big\{t\geq 0: \varPsi_t \geq b_{\a_t}(\varPhi_t)\Big\},\ea\eea
	where   the sufficient statistics $(\a,\varPhi,\varPsi)$ are given by
	$$\left\{\ba{ll}\ad L^i_t=\exp\Big\{\int_0^t\mu(\a_s)dX^i_s-\frac12\int_0^t\mu^2(\a_s)ds\Big\}\text{ for }i=1,2;\\\ns\ad\varPhi_t=e^{(\lambda+\gamma)t}L^1_t \(\frac\pi{1-\pi}+\lambda\int_0^t\frac{1}{e^{(\lambda+\gamma)s}L^1_s}ds\);\\
	\ns\ad\varPsi_t=e^{(\lambda+\gamma)t}L^2_t \(\frac\pi{1-\pi}+\lambda\int_0^t\frac{1}{e^{(\lambda+\gamma)s}L^2_s}ds\),\ea\right.$$
	%
	where $\a$ is a Markov chain with initial state $\a_0=\iota$ and $\{b_\iota\}_{\iota\in M}$ is the unique solution to \eqref{bouneq} in the admissible class given in Theorem \ref{mainthm}.
\end{corollary}

\noindent \textbf{Acknowledgements} The first-named author is grateful to ARO-YIP-71636-MA, NSF DMS-1811936,  ONR
N00014-18-1-2192, and ONR N00014-21-1-2672 for their support of this research.


\begin{center}
	
\end{center}

\end{document}